\documentclass[12pt]{amsart}

\usepackage{amsmath}
\usepackage{subcaption}
\usepackage{enumerate}

\usepackage{graphicx}
\usepackage[utf8]{inputenc}
\usepackage{amsmath}
\usepackage{amsfonts}
\usepackage{mathtools}
\usepackage{thmtools}
\usepackage{hyperref}       
\usepackage{tikz}
\usetikzlibrary{shapes}
\usetikzlibrary{positioning}
\usepackage{tkz-euclide}
\usetikzlibrary{through,intersections}

\calclayout

\newcommand*{\R}{\mathbb{R}}

\newcommand{\probP}{\text{I\kern-0.15em P}}

\usepackage{thmtools}
\usepackage{amsthm}
\usepackage[T1]{fontenc}
\usepackage[toc]{appendix}
\usepackage{tikz}

\title{Dual systolic graphs}
\author{Daniel Carmon \and Amir Yehudayoff}

\date{}
\begin{document}
\theoremstyle{plain}
\newtheorem{theorem}{Theorem}
\newtheorem{lemma}[theorem]{Lemma}
\newtheorem{claim}[theorem]{Claim}

\newtheorem*{definition}{Definition}

\maketitle

\begin{abstract}
We define a family of graphs we call {\em dual systolic} graphs.
This definition comes from graphs that are duals of systolic
simplicial complexes.
Our main result is a sharp (up to constants) isoperimetric inequality for dual systolic graphs.
The first step in the proof is an extension of the classical 
isoperimetric inequality of the boolean cube.
The isoperimetric inequality for dual systolic graphs, however,
is exponentially stronger than the one for the boolean cube. 

Interestingly, we know that dual systolic graphs exist,
but we do not yet know how to efficiently construct them.
We, therefore, define a weaker notion of dual systolicity.
We prove the same isoperimetric inequality for weakly dual systolic graphs,
and at the same time provide an efficient construction of a family of graphs that are weakly dual systolic.
We call this family of graphs {\em clique products}.
We show that there is a non-trivial connection between
the small set expansion capabilities and the threshold rank of clique products,
and believe they can find further applications.

%
\end{abstract}

\section{Introduction}

The systole of a space is the smallest length of a cycle that can not be
contracted.
Systoles were originally studied for manifolds,
but our focus is combinatorial.
For graphs, the systole is the girth.
A graph is systolic if it locally looks like a tree.
In a broader topological setting,
a simplicial complex\footnote{A collection of finite sets
that is closed downwards.} is systolic 
if every short cycle in its one-skeleton is contractible
(roughly speaking, there are no short induced cycles).
The systolic condition for simplicial complexes can be thought of as 
a combinatorial definition of non-positive curvature.
Systolic inequalities play an important role in many areas of mathematics
(see~\cite{katz2007systolic,gromov1992systoles} and references within).

%


A particular systolic condition for simplicial complexes, 
identified by Davis and Moussong, is the {\em no empty square} condition
(see~\cite{moussong1988hyperbolic,davis2002nonpositive,Janus:2003} and references within).
This condition implies that the one-skeleton does not contain an induced cycle of length 
at most four. 
It was open whether simplicial complexes satisfying this condition
exist or not.
Januszkiewicz and Świątkowski were the first to construct
such complexes~\cite{Janus:2003,Janus:2006}.
Their construction is deep and sophisticated.
Its importance is also evident from
its application in machine learning~\cite{brukhim:2022}.
This application highlights the fundamental importance of systoles.
Roughly speaking, some systolic spaces lead
to learning problems that are difficult to solve
due to a ``global'' reason but not a ``local'' reason.

The starting point of our work was to improve our understanding of
the JS constructions.
A specific question that came up is
{\em what is special about duals of systolic spaces?}
Every pure\footnote{All maximal simplexes
have the same dimension.} simplicial complex defines a dual graph. 
The vertices of the graph are the facets\footnote{Full-dimensional
simplexes.} in the complex.
Two vertices are connected by an edge
if the two facets share a face of co-dimension one.
The fact that the graph is a dual of a complex 
automatically leads to non-trivial behavior.
We formalize such behavior using the notion of pseudo-cubes.

\begin{definition}
We say that a graph $G = (V,E)$ is a {\em pseudo-cube} if the following conditions are
satisfied:
\begin{enumerate}[(i)]
\item It is $d$-regular.
\item It has an edge $d$-coloring; that is, there is $\chi :E \to [d]$
so that for all $v \in V$,
if $E_v$ is the set of $d$ edges touching the vertex $v$,
then $\chi(E_v) = [d]$.
\item For every color $i \in [d]$,
the graph obtained from $G$ by contracting all edges of color~$\neq i$
does not have self-loop.
\end{enumerate}
\end{definition}

The number~$d$ is called the {\em dimension} of the pseudo-cube.
Property (iii) is called {\em color independence}.
In other words, it means that for every color $i$
and every cycle in $G$,
the number of edges of color $i$ along the cycle 
is not one (either zero or at least two).

While this definition is simple to state, 
let us describe some examples, and provide some intuition.
A well-known example of a pseudo-cube is the $d$-dimensional boolean cube~$Q_d$.
Its vertex set is $\{0,1\}^d$
and two vertices are connected by an edge if they differ in a single entry.
The edge coloring comes from the $d$ dimensions or directions in the cube.
The color independence follows from the linear independence
of the standard basis.

The definition of a pseudo-cube graph corresponds to the graph being dual to a 
{\em chromatic} and {\em non-branching} pure complex.
A simplicial complex is {\em chromatic} if there is a coloring of its vertices 
such that each facet has a vertex of each color. 
A pure simplicial is {\em non-branching} if every 
simplex of co-dimension one is a face of exactly two facets.
Such simplicial complexes allow us to encode data in a useful manner.
The general theme is to model ``legal or possible states'' as topological spaces.

Consider, for example, the following scenario.
There are three players $\textcolor{blue!95}{A},\textcolor{violet!95}{B},\textcolor{teal!95}{C}$, and a deck of four cards $1,2,3,4$. 
Each player gets a unique card.
There are $24$ possible assignments of cards to players.
This scenario can be modeled as a chromatic and non-branching simplicial complex
(see Figure~\ref{fig:cards} and Example~3 in~\cite{Goubault_2018}).
A vertex in the complex corresponds to the local view of a single player
(i.e., the vertices are pairs of the form $(\textcolor{blue!95}{A},1)$). 
The color of a vertex is the player's name (the color of $(\textcolor{blue!95}{A},1)$ is $\textcolor{blue!95}{A}$).
The $24$ facets correspond to possible states of the world such as $\{(\textcolor{blue!95}{A},2),(\textcolor{violet!95}{B},1),(\textcolor{teal!95}{C},4)\}$.
Each facet has three colors. 
When a single player switches between her card and the free card,
we move to an adjacent facet.
This simplicial complex turns out to be a triangulation of the two-dimensional torus.

\begin{figure}[h!]
\centering

\begin{tikzpicture} 
\tikzset{bluecirc/.style={circle, draw=black, fill=blue!50, inner sep=1pt,minimum size=8pt}}
\tikzset{violsqur/.style={circle,draw=black, fill=violet!50, inner sep=1pt,minimum size=8pt}}
\tikzset{tealtri/.style={circle, draw=black, fill=teal!50, inner sep=1pt,minimum size=8pt}}


\draw [fill=gray!50] (2,0) -- (8,0) -- (9,1.73) -- (3,1.73) -- cycle;

\foreach \x in {2.5,...,9}{
    \foreach \y in {1}{
      \pgfmathsetmacro\xx{\x - 0.5*Mod(\y+1,2)};
      \pgfmathsetmacro\yy{\y*sqrt(3)/2};
      \pgfmathsetmacro\yyy{(\y+1)*sqrt(3)/2};
      
     \draw (\xx,\yy) -- ++ (60:1); 
     
    }
}

\foreach \x in {2.5,...,7.5}{
    \foreach \y in {1}{
      \pgfmathsetmacro\xx{\x - 0.5*Mod(\y+1,2)};
      \pgfmathsetmacro\yy{\y*sqrt(3)/2};
      \pgfmathsetmacro\yyy{(\y+1)*sqrt(3)/2};
      
     \draw (\xx,\yy) -- ++ (-60:1); 
     
    }
}

\foreach \x in {2.5,...,8}{
    \foreach \y in {1}{
      \pgfmathsetmacro\xx{\x - 0.5*Mod(\y+1,2)};
      \pgfmathsetmacro\yy{\y*sqrt(3)/2};
      \pgfmathsetmacro\yyy{(\y+1)*sqrt(3)/2};
      
     \draw (\xx,\yy) -- ++ (0:1); 
     
    }
}

\foreach \x in {2.5,...,8}{
    \foreach \y in {0}{
      \pgfmathsetmacro\xx{\x - 0.5*Mod(\y+1,2)};
      \pgfmathsetmacro\yy{\y*sqrt(3)/2};
      \pgfmathsetmacro\yyy{(\y+1)*sqrt(3)/2};
      
     \draw (\xx,\yy) -- ++ (0:1); 
     
    }
}

\foreach \x in {3.5,...,9}{
    \foreach \y in {2}{
      \pgfmathsetmacro\xx{\x - 0.5*Mod(\y+1,2)};
      \pgfmathsetmacro\yy{\y*sqrt(3)/2};
      \pgfmathsetmacro\yyy{(\y+1)*sqrt(3)/2};
      
     \draw (\xx,\yy) -- ++ (0:1); 
     
    }
}

\foreach \x in {3.5,...,9}{
    \foreach \y in {2}{
      \pgfmathsetmacro\xx{\x - 0.5*Mod(\y+1,2)};
      \pgfmathsetmacro\yy{\y*sqrt(3)/2};
      \pgfmathsetmacro\yyy{(\y+1)*sqrt(3)/2};
      
     \draw (\xx,\yy) -- ++ (-60:1); 
     
    }
}

\foreach \x in {2.5,...,9}{
    \foreach \y in {0}{
      \pgfmathsetmacro\xx{\x - 0.5*Mod(\y+1,2)};
      \pgfmathsetmacro\yy{\y*sqrt(3)/2};
      \pgfmathsetmacro\yyy{(\y+1)*sqrt(3)/2};
        \pgfmathsetmacro\xxx{int(1+3*Mod(\x-0.5,2))};

     \draw (\xx,\yy) -- ++ (60:1); 
    
                \pgfmathparse{int(mod(\x,3))} 
             \ifnum0=\pgfmathresult\relax
                     \node[violsqur]  at (\xx,\yy) {$\xxx$};                           
                \else  
                 \ifnum1=\pgfmathresult\relax
                     \node[tealtri]   at (\xx,\yy) {$\xxx$};
                \else  
                    \node[bluecirc]  at (\xx,\yy) {$\xxx$};
                \fi
               \fi
    }
}

\foreach \x in {2.5,...,9}{
    \foreach \y in {0}{
      \pgfmathsetmacro\xx{\x - 0.5*Mod(\y+1,2)};
      \pgfmathsetmacro\yy{\y*sqrt(3)/2};
      \pgfmathsetmacro\yyy{(\y+1)*sqrt(3)/2};
        \pgfmathsetmacro\xxx{int(2+Mod(\x+0.5,2))};

                \pgfmathparse{int(mod(\x,3))} 
             \ifnum0=\pgfmathresult\relax
                    \node[bluecirc]  at (\x,\yyy) {$\xxx$};
                \else  
                 \ifnum1=\pgfmathresult\relax
                     \node[violsqur]  at (\x,\yyy) {$\xxx$};
                \else  
                    \node[tealtri]  at (\x,\yyy) {$\xxx$};
                \fi
               \fi
    }
}

\foreach \x in {3.5,...,10}{
    \foreach \y in {2}{
      \pgfmathsetmacro\xx{\x - 0.5*Mod(\y+1,2)};
      \pgfmathsetmacro\yy{\y*sqrt(3)/2};
      \pgfmathsetmacro\yyy{(\y+1)*sqrt(3)/2};
              \pgfmathsetmacro\xxx{int(1+3*Mod(\x+0.5,2))};

    
                \pgfmathparse{int(mod(\x,3))} 
             \ifnum0=\pgfmathresult\relax
                     \node[violsqur]  at (\xx,\yy) {$\xxx$};                           
                \else  
                 \ifnum1=\pgfmathresult\relax
                     \node[tealtri]   at (\xx,\yy) {$\xxx$};
                \else  
                    \node[bluecirc]  at (\xx,\yy) {$\xxx$};
                \fi
               \fi
    }
}

\end{tikzpicture}
\caption{Four cards to three players as a simplicial complex.
The 24 triangles are facets,
and vertices of the same color and number
are identified.}\label{fig:cards}
\end{figure}
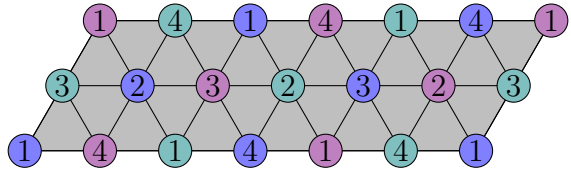

The translation to the topological language leads to
important results in computer science. 
A prominent example is the proof of the asynchronous computability theorem in ~\cite{herlihy1999topological}, where the different colors corresponded to different processors in a distributed setting.
In the recent work~\cite{brukhim:2022}, 
such simplicial complexes were used
to model multiclass learning problems.

The following example provides some additional intuition.
Imagine the unit interval $I=[0,1]$ placed on the line $\R$.
Let $x$ be the reflection around $1 \in \R$
and let $y$ be the reflection around $0 \in \R$.
Both $x$ and $y$ are involutions.
If we apply $x$ on $I$ we get $xI = [1,2]$,
and if we apply $y$ we get $yI = [-1,0]$.
The group $D$ generated by $x$ and $y$ is the infinite dihedral group.
If we apply all of $D$ on $[0,1]$, we get a tiling of the line by unit intervals.
This tiling defines a simplicial complex. 
The dual graph is the Cayley graph of $D$ with $x,y$ as generators,
and is a pseudo-cube.
The intermediate value theorem says that if we move continuously in the line
starting at some unit interval and coming back to the same interval,
then we must cross the same boundary of the interval more than once.
This topological property is abstracted by color independence.

The boolean cube is the dual graph of a non-branching and chromatic 
pure complex.
The vertices of this complex are the $2d$ elements in $[d] \times \{0,1\}$.
Each $x \in \{0,1\}^d$ defines a full dimensional simplex $\sigma_x = \{(i,x_i): i \in [d]\}$.
Two vertices $x,y \in \{0,1\}^d$ are connected by an edge 
iff the common face $\sigma_x \cap \sigma_y$
has co-dimension one.
This simplicial complex is, however, not systolic.
For example, for $d=2$, this complex is an empty square. 

The JS complexes are systolic.
Their dual graphs are therefore {\em dual systolic}.
The following definition provides a combinatorial abstraction.

\begin{definition}{\label{dualsystolic}}
A graph $G$ is {\em dual systolic} of dimension $d$
if it is a pseudo-cube of dimension $d$ so that for every color $i \in [d]$,
the graph obtained from $G$ by contracting all edges of color~$\neq i$
is simple (that is, there are no double edges between vertices).
\end{definition}

Stated differently, dual systolicity means that for every color $i$, 
the following holds.
Let $V_1,V_2 \subseteq V$ be two distinct connected components
with respect to edges of color $\neq i$.
There is at most a single edge of color $i$ between $V_1$ and $V_2$ in $G$.

{\em How does this definition relate to systoles?} 
The observation is that if a simplicial complex
has no empty squares then its dual graph is dual systolic.
Let $\mathcal{C}$ be a chromatic, non-branching and pure simplicial complex,
and let $G$ be its dual graph.
The {\em star} of a vertex $v$ in $\mathcal{C}$ is the set of all simplexes 
$\sigma \in \mathcal{C}$ so that $\{v\} \cup \sigma \in \mathcal{C}$.
The vertex $v$ is called the center of the star.
There is a one-to-one correspondence between 
stars in $\mathcal{C}$ whose center has color $i$ and
connected components in the graph obtained
by deleting edges with color $i$ from $G$.
Let $v_1,v_2$ be two distinct vertices of color $i$ in $\mathcal{C}$.
Assume that the two corresponding connected components have two or more $i$-edges between them.
It follows that the corresponding stars share two different simplexes $\sigma_1,\sigma_2$ of co-dimension one. 
There are two distinct vertices $u_1 \in \sigma_1$
and $u_2 \in \sigma_2$ with the same color.
The square $v_1,u_1,v_2,u_2$ has alternating colors so it is empty
(because the complex is chromatic).

\begin{figure}[t]
  \centering
  \begin{subfigure}[b]{0.3\textwidth}
    \includegraphics[width=\textwidth]{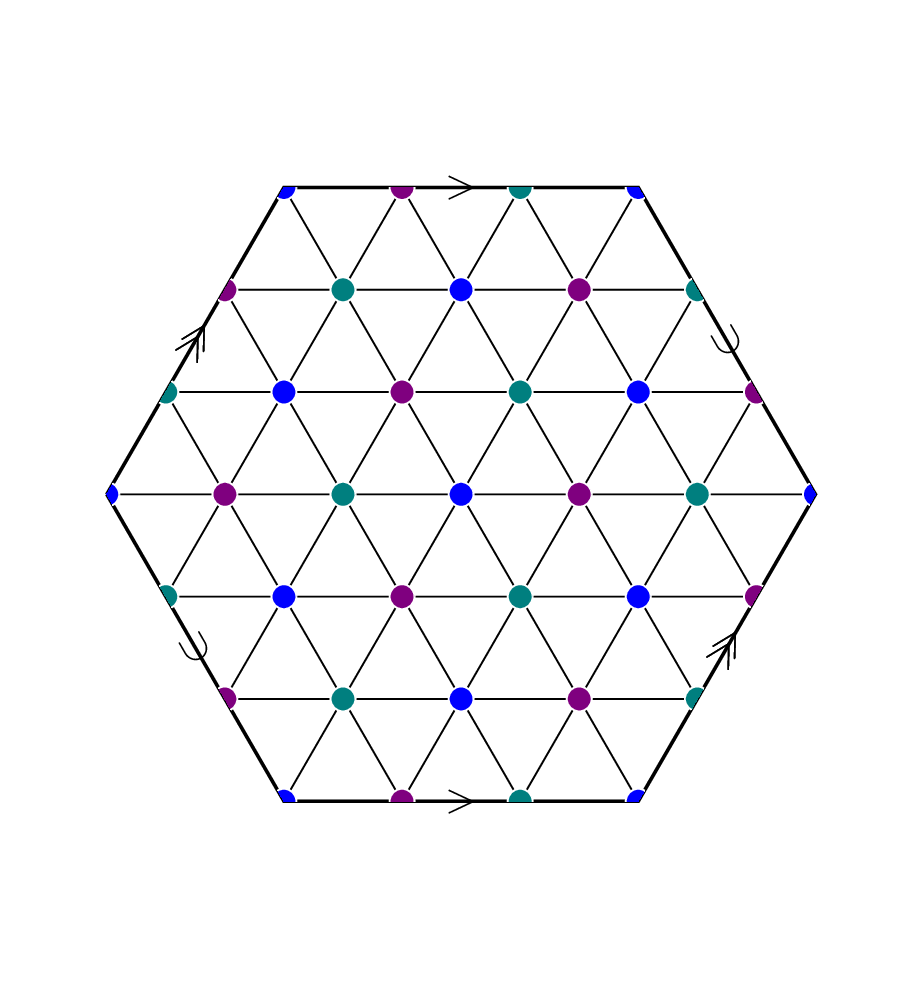}
    \label{fig:image1}
  \end{subfigure}
  \hfill
    \begin{subfigure}[b]{0.3\textwidth}
    \includegraphics[width=\textwidth]{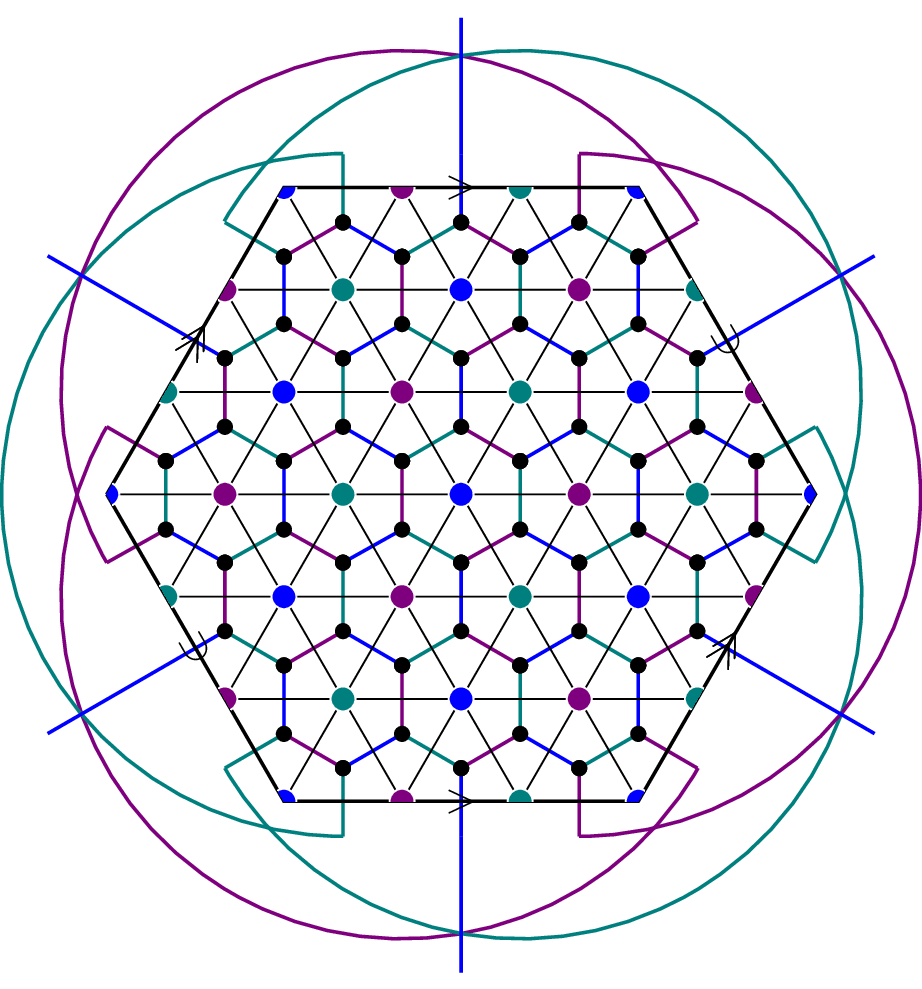}
  \label{fig:image3}
	\end{subfigure}
  \hfill
  \begin{subfigure}[b]{0.3\textwidth}
    \includegraphics[width=\textwidth]{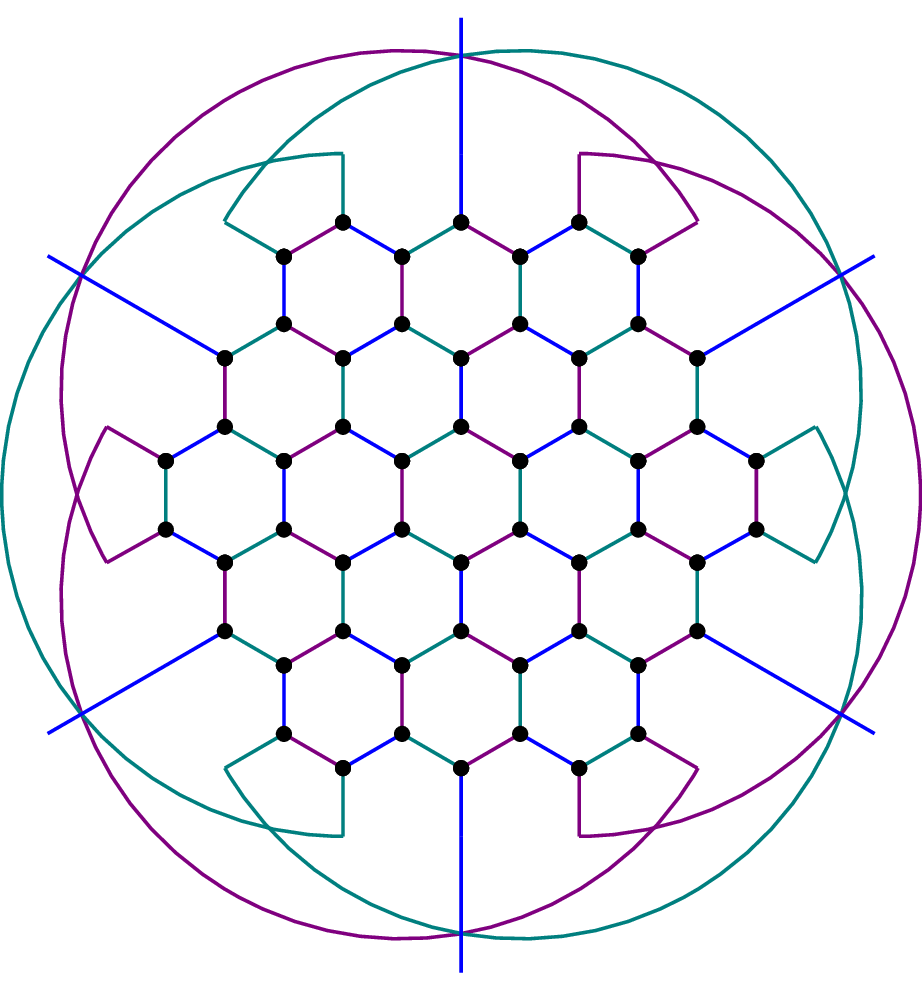}
    \label{fig:image2}
  \end{subfigure}
  \caption{A systolic complex (on the left) and its dual graph (on the right).}

\end{figure}

The $d$-dimensional boolean cube $Q_d$ is not dual systolic for $d>1$;
for every $i \in [d]$, the 
the graph obtained from $Q_d$ by contracting all edges of color~$\neq i$
consists of $2^{d-1}$ parallel edges. 
In a sense, it is the ``least dual systolic'' pseudo-cube.
The dual graphs of the JS complexes are dual systolic.

The definition of dual systolic graphs is simple and natural.
The constructions of Januszkiewicz and Świątkowski imply that
dual systolic graphs exist~\cite{Janus:2003,Janus:2006}.
It seems, however, quite complicated to efficiently build dual systolic graphs.
For example, we do not yet know of strongly polynomial time algorithms
for building them.

\begin{figure}[t]
\includegraphics[width=4cm]{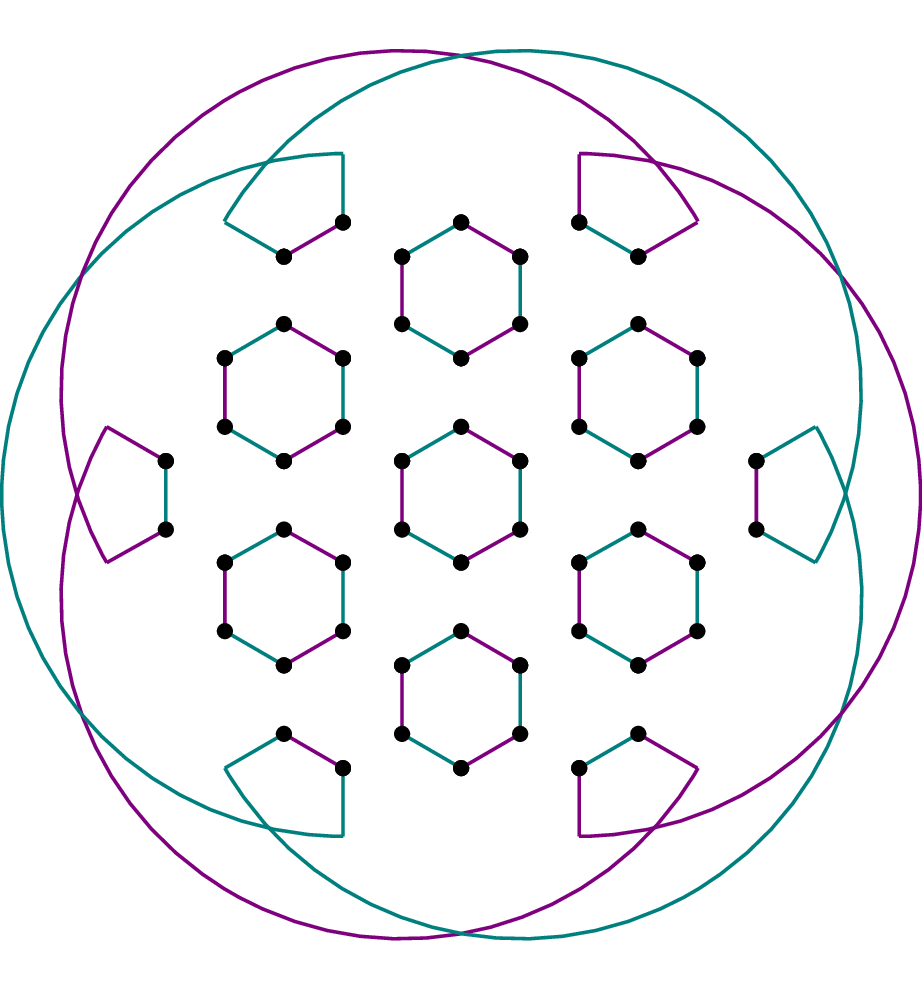}
  \caption{The dual systolic graph from Figure 1 with blue edges deleted.}
\end{figure}

\subsection{Isoperimetry}

The first property we study is the isoperimetric profile of dual systolic graphs.
Let $G  = (V,E)$ be a graph.
For $U \subseteq V$,
denote by $\partial(U) = \partial_G(U)$ the number of edges 
connecting $U$ and $V \setminus U$ in $G$.
The {\em edge expansion} of $U \neq \emptyset$ is
$$\phi(U) = \phi_G(U) = \frac{\partial(U)}{|U|}.$$
The isoperimetric profile of $G$ is the function defined by 
$$P(s) = P_G(s) = \min \{ \phi(U) : |U|=s\}.$$
It provides a lot of information on the structure of $G$.

The first result we state bounds the isoperimetric profile of pseudo-cubes.
Samorodnitsky proved the same result for the boolean cube~\cite{Samorodnitsky2008AML}.
We observe that his proof can be extended to all pseudo-cubes.\footnote{Here and below, logarithms are in base two.}

\begin{theorem}
For every $d$-dimensional pseudo-cube $G$ and $s>0$,
$$P_G(s) \geq d - \log s.$$
\end{theorem}

The theorem is sharp for the boolean cube for all integers $s$
that are powers of two.
It can be thought of as saying that pseudo-cubes are small set expanders.
The study of small set expansion is motivated by its intimate connection with the unique games conjecture, its potential to provide insights into hardness of optimization problems, 
and its connection to metric embeddings 
(see~\cite{Stuerer:2010,Subhash:2015,5671307} and references within). 


The boolean cube is the canonical example of a small set expander. 
The small set expansion of various other graphs has been studied in several works. Noisy versions of the boolean cube were studied in~\cite{Kahn:1988,Mossel2005NoiseSO}. 
Small set expansion for noisy cubes is related to 
the "majority is stablest" theorem.
The authors of~\cite{BarakGHMRS15} considered a derandomized version of the noisy cube
and achieved a better trade-off between expansion and threshold rank (more on this later on).
The small set expansion of the multi-slice graph was studied in~\cite{Filmus:2018},
and of the Johnson graph was studied in~\cite{khot2018small}.
Another line of research includes characterizing the non-expanding subsets and showing that they are, in some sense, negligible~\cite{grassman}. 

Dual systolicity leads to a much stronger
bound on the isoperimetric profile. 

\begin{theorem}\label{theta_loglog}
    For every $d$-dimensional dual systolic graph $G$ and $s>1$,
    \begin{equation*}
    P_G(s) \geq d-8 (1+\log\log s ) .
    \end{equation*}
\end{theorem}

This is an exponential improvement over the previous theorem. 
Pseudo-cubes have strong expansion, say of at least $\tfrac{d}{2}$,
for sets of size at most exponential in $d$.
The same expansion for dual systolic graphs holds
for sets of size {\em doubly exponential} in $d$.
This bound is sharp (up to the constant before the $\log \log$)
because the JS construction yields dual systolic graphs
of size doubly exponential in $d$.

Most proofs of small set expansion are analytic and go through
hypercontractivity. 
Our proof follows a different path.
The bound for pseudo-cubes uses information theory
(following Samorodnitsky's footsteps).
The bound for dual systolic graphs has three parts.
The first part is the bound for pseudo-cubes.
The second part is identifying a dynamics
that the isoperimetric profile satisfies.
On a high-level, there is a functional $\mathcal F$ so that
if $P$ is the isoperimetric profile of dual systolic graphs,
then ${\mathcal F}(P)$ is also such a profile.
The third part is finding the fixed point of this dynamics.
This dynamics can be thought of as a bootstrapping 
mechanism for proving isoperimetric inequalities 
that relies on dual systolicity.
For more details, see Section~\ref{sec:dualSys}.

\subsection{Explicit constructions}

The JS complexes are defined via a collection of groups.
This collection of groups is inductively constructed,
but we do not know of an efficient algorithm that implements this construction.
In this section, we weaken the dual systolic condition.
This leads to explicit constructions and at the same time
we get the same isoperimetric behavior as for dual systolic graphs. 

The first notion we define is that of a {\em weak} pseudo-cube.
We replace color independence with a weaker requirement.

\begin{definition}
We say that a graph $G = (V,E)$ is a {\em weak} pseudo-cube if
(i) it is $d$-regular, (ii) it has an edge $d$-coloring, and (iii)
the following two conditions hold.
First, the graph obtained from $G$ by contracting all edges of color $\neq d$
does not have self-loop.
Second, let $G_{-d}$ denote the graph obtained from $G$ by deleting all
edges of color $d$.
Every connected component in $G_{-d}$ is a $(d-1)$-dimensional pseudo-cube.
This definition is inductive;
for $d=1$, a weak pseudo-cube is a perfect matching.
\end{definition}

Property (iii) is called {\em weak} color independence.
Weak pseudo-cubes satisfy the same isoperimetric inequality
as the one stated for pseudo-cubes above.

\begin{theorem}
\label{thm:weakIso}
For every $d$-dimensional weak pseudo-cube $G$ and $s>0$,
$$P_G(s) \geq d - \log s.$$
\end{theorem}

We now define the weak notion of dual systolicity.

\begin{definition}
A graph $G$ is {\em weakly dual systolic}
if it is a weak pseudo-cube of dimension $d$ so that
(a) the graph obtained from $G$ by contracting all edges of color~$\neq d$
is simple,
and (b) every connected component in $G_{-d}$ is weakly systolic
of dimension $d-1$.
This definition is inductive;
for $d=1$, a weakly systolic graph is a perfect matching.
\end{definition}

Weakly dual systolic graphs also satisfy the same strong
isoperimetric inequality as dual systolic graphs. 

\begin{theorem}\label{theta_loglog_weak}
    For every $d$-dimensional weakly dual systolic graph $G$ and $s>1$,
    \begin{equation*}
    P_G(s) \geq d-8 (1+\log\log s) .
    \end{equation*}
\end{theorem}

Even though we weakened the requirements 
we obtained the same isoperimetric behavior.
But we also obtained one more important thing.
There are explicit and simple constructions of weakly dual systolic
graphs. 

Before describing the construction, we answer a basic question. 
{\em What is the least size of a weakly systolic graph of dimension $d$?}
The size of $d$-dimensional pseudo-cubes is at least roughly $\exp(d)$.
The size of $d$-dimensional dual systolic graphs is much larger,
at least roughly $\exp(\exp(d))$.
Indeed, if we denote by $F(d)$ the minimal number of vertices 
in a $d$-dimensional weakly dual systolic graph, we have the recurrence relation:
$F(1) = 2$ and for $d>1$,
\begin{equation}\label{simple_rec}
    F(d) \geq F(d-1) \big( F(d-1)+1\big) .
\end{equation}
The inequality holds because if we look at a single connected component of $G_{-d}$
of size $S$ then it has $S$ neighboring components,
and the size of each of these $S+1$ components is at least $F(d-1)$.

\begin{claim}
\label{clm:H}
For every integer $d \geq 1$,
there is an explicit weakly dual systolic graph $CP^{(d)}$ 
of dimension $d$ with $n^{(d)}$ vertices.
In addition $n^{(1)}=2$ and $n^{(d)} = n^{(d-1)} ( n^{(d-1)}+1)$ for $d>1$.
\end{claim}

\begin{figure}[t]
\centering
\subfloat{
\includegraphics[width=0.2\textwidth]{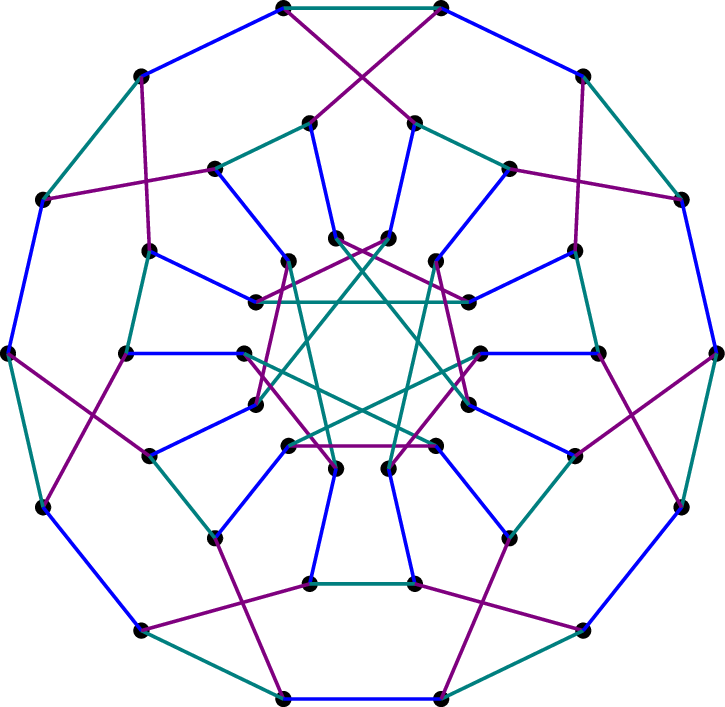}
}
\subfloat{
\includegraphics[width=0.2\textwidth]{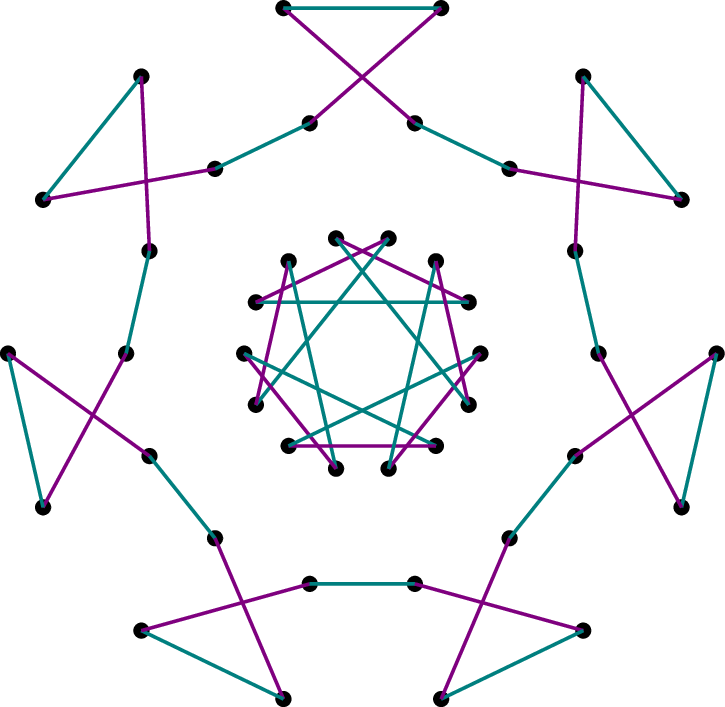}
}
\\
\vspace{0.1cm}
\subfloat{
\includegraphics[width=0.2\textwidth]{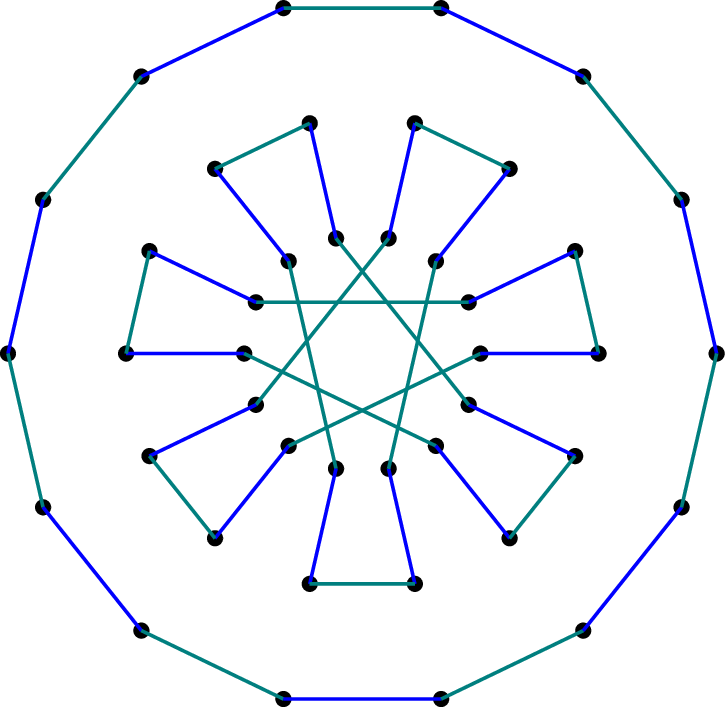}
}
\subfloat{
\includegraphics[width=0.2\textwidth]{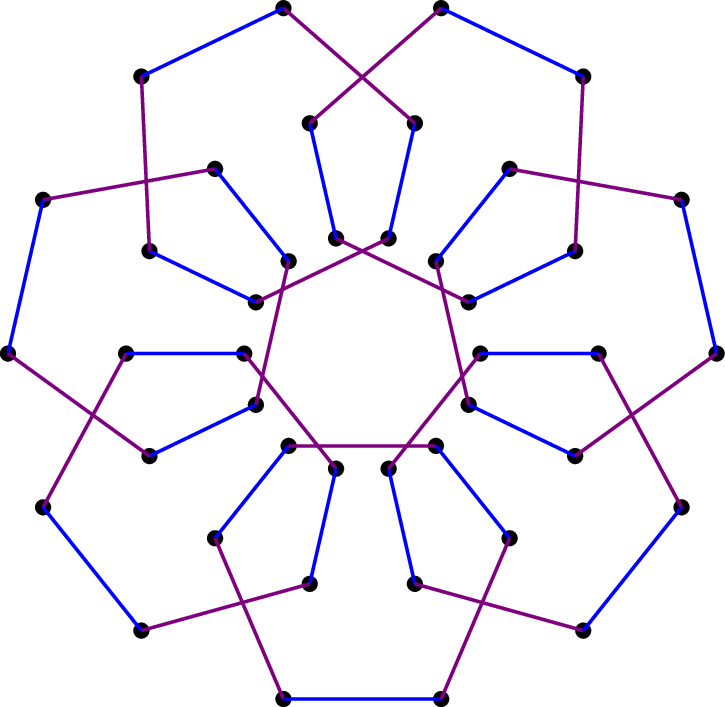}
}
\caption{The graph $CP^{(3)}$ is displayed in the top-left corner. The other three images show this graph when we delete edges of a specific color. Only the green color 
on the bottom-right ``partitions the graph'' in the desired way.}
\label{fig:clique_prod}

\end{figure}

The graph $CP^{(d)}$ is called the {\em $d$-dimensional clique product.}
The construction of clique products is by repeated applications of replacement products of cliques~\cite{wigderson2002entropy,Hoory2006ExpanderGA}. 
See Figure~\ref{fig:clique_prod} for an example.

The clique products show that inequality~\eqref{simple_rec} is
in fact an equality. 
The $d$-dimensional clique product is the 
smallest possible $d$-dimensional weakly systolic graph.
In some sense, clique products are the dual systolic versions of the boolean cubes
(which are the smallest possible pseudo-cubes).
Below we describe one application of clique products.
It seems reasonable that they can be used in future
applications as well.

\begin{proof}[Proof of Claim~\ref{clm:H}]
Take $CP^{(1)}$ to be a single edge. 
Construct $CP^{(d)}$ from $CP^{(d-1)}$ by taking its replacement product with 
the clique of size $n^{(d-1)}+1$.
Specifically, take $n^{(d-1)} +1$ disjoint copies of $CP^{(d-1)}$, 
and connect node $i \in \{1,2,\ldots,n^{(d-1)}\}$ 
in copy $j \in \{0,1,\ldots,n^{(d-1)}\}$ to node $-i$ in copy $i+j$,
where the operations are modulo $n^{(d-1)}+1$. 
Color all these new edges with color~$d$.
\end{proof}

\subsection{An application} 

The following question 
came up in the study of the unique games conjecture:
{\em what is the relationship between threshold rank
and small set expansion?}

Let $G$ be a regular graph with $n$ vertices, and let $M$ be its normalized adjacency matrix.
The matrix $M$ is symmetric and has $n$ eigenvalues
(the maximum eigenvalue is one).
The {\em threshold rank} of $G$ with parameter $\epsilon>0$,
denoted by $TR_{1-\epsilon}(G)$, is defined
to be the number of eigenvalues of $M$ that are at least $1-\epsilon$.

Threshold rank is important in algorithms for the unique games conjecture~\cite{Kolla10,5671307}.
The authors of~\cite{5671307} bounded from above 
the threshold rank of small set expanders.
The authors of~\cite{BarakGHMRS15} constructed small
set expanders with relatively large threshold rank.
A full understanding of the relationship between threshold rank 
and small set expansion remains an open problem.
The clique products are (somewhat) small
set expanders with high threshold rank.


\begin{theorem}
\label{thm:trhRk}
Let $d>2$, let $CP = CP^{(d)}$ and let $n = n^{(d)}$.
Let $0 < k \leq \tfrac{d}{2}$ be an integer and
$\epsilon \geq \frac{2k}{d}$.
Then,
\begin{equation*}
TR_{1-\epsilon}(CP) \geq \frac{n^{1-2^{-k}}}{2}. 
\end{equation*}
\end{theorem}

For $k = 1$, we get a lower bound of $\Omega\left(\sqrt{n}\right)$ on the threshold rank
with $\epsilon = \tfrac{2}{d}$.
In other words, 
we have a graph of size $n$
so that the threshold rank $TR_{1-\epsilon}$ is polynomially large,
for $\epsilon \leq \frac{2}{1+\log\log n}$ that tends to zero as $n$ tends to infinity.
When we think of $\epsilon$ as a constant, the integer
$k$ becomes $\tfrac{\epsilon d}{2}$,
and the threshold rank is almost as high as it can be 
($TR_{1-\epsilon}$ is at least $n^{1-o(1)}$). 

A rough comparison can be made for example with Theorem~4.14 in~\cite{BarakGHMRS15}. Their construction depends on two parameters. 
If we plug-in $\epsilon$ as one of their parameters,
and $\tfrac{1}{2}$ as the second parameter, 
we see that the threshold rank of $CP^{(d)}$ is larger than in their construction, 
but the small set expansion properties in their construction is better.
The two constructions are incomparable.

The ideas developed in our paper can perhaps lead to better constructions
in the small set expansion versus threshold rank question.
A specific approach that seems reasonable
is weakening the dual systolicity condition;
instead of requiring at most one edge between two connected components,
we can upper bound the number of edges between components by some other number.
This weakening still leads to small set expansion behavior,
and can potentially lead to a wider range of parameters. 

\subsection{A discussion}

The dual systolicity condition translates 
the systolic notion from~\cite{Janus:2003,Janus:2006} to the language of graphs.
There are stronger systolic criteria for simplicial complexes than the one considered here 
(i.e., the $k$-large condition). 
We believe that further exploring these stronger criteria could lead to 
stronger guarantees and to powerful properties.
Our work is only a first step in this direction.

\section{Isoperimetry of weak pseudo-cubes}

In this section we prove bounds on the isoperimetric profile of weak pseudo-cubes.
    The proof we present here uses information theory, and is 
    inspired by Samorodnitsky's proof for the hypercube~\cite{Samorodnitsky2008AML}.
    The main difficulty we face is identifying the correct ``coordinate system'' to use.

\begin{proof}[Proof of Theorem~\ref{thm:weakIso}]
 Let $G = (V,E)$ be a $d$-dimensional weak pseudo-cube,
 and let $U$ be a non-empty set of vertices.    
    For $i \in [d]$, denote all inner $i$-edges of $U$ by
    \begin{equation*}
        e_i (U):=\{e \in E : e \subseteq U , \chi(e)=i\} .
    \end{equation*}
   Because
    \begin{equation*}
        d|U| = \partial(U) + 2\sum_{i \in [d]} e_i(U) ,
    \end{equation*}
we see that
    \begin{equation*}
        \phi(U) = d - \frac{2\sum_{i \in [d]} e_i(U)}{|U|} .
    \end{equation*}
It remains to prove that 
    \begin{equation}\label{comp1}
        \frac{2\sum_{i=1}^d e_i(U)}{|U|} \leq \log |U| .
    \end{equation}
    Let $X$ be a uniformly random vertex in $U$
    so that the Shannon entropy of $X$ is $H(X) = \log |U|$. 
    For every $I\subseteq [d]$, 
    denote by $X_I$ the set
    \begin{equation*}
        X_I = \{v\in V : \text{ $v$ is reachable from $X$ using colors in $I$}\} .
    \end{equation*}
   The crucial observation is that for all $e,i$ so that $\Pr[X_{\{i\}} = e] > 0$,
   we have
$$H(X|X_{\{i\}} = e) = 1_{e\in e_i(U)}.$$
In other words, when $e \in e_i(U)$, the entropy is one,
and when $e \not \in e_i(U)$, the entropy is zero.
Taking expectation, we see that
$$H(X|X_{\{i\}}) = \frac{2 e_i (U)}{|U|}.$$
Inequality~\eqref{comp1} can thus be written as
    \begin{equation*}
        \sum_{i \in [d]} H(X|X_{\{i\}}) \leq H(X) .
    \end{equation*}
To show this, we observe that for all $i > 1$,
weak color independence implies that
$$H(X|X_{\{i\}}) \leq H(X_{\{1,..,i-1\}}|X_{\{i\}}).$$
When $H(X|X_{\{i\}} = e)$ is one,
the distribution of $X_{\{1,..,i-1\}}$ conditioned on $X_{\{i\}} = e$ 
is uniform between two options
(so that the r.h.s.\ is one as well).
The set $X_I$ is determined by the set $X_{I'}$ whenever $I'\subseteq I$.
The data processing inequality thus implies
that for all $i>1$,
$$H(X_{\{1,\ldots,i-1\}}|X_{\{i\}}) \leq H(X_{\{1,\ldots,i-1\}}|X_{\{1,2,\ldots,i\}}).$$
By the chain rule,
 \begin{align*}
        \sum_{i \in [d]} H(X|X_{\{i\}}) &= H(X|X_{\{1\}})+\sum_{i>1} H(X|X_{\{i\}})\\
        &\leq H(X|X_{\{1\}}) +\sum_{i>1} H(X_{\{1,\ldots,i-1\}}|X_{\{1,2,\ldots,i\}})\\
        &= H(X,X_{\{1\}},X_{\{1,2\}},\ldots,X_{\{1,\ldots,d-1\}}|X_{\{1,\ldots,d\}})\\
        & \leq H(X) . \qedhere
    \end{align*}
\end{proof}
\section{Isoperimetry for weak dual systolic graphs}
\label{sec:dualSys}

This section is devoted for analyzing the isoperimetric profile
of weakly dual systolic graphs.

\subsection{A dynamic}

The key step in the proof is identifying a dynamic that
the isoperimetric profile satisfies.
This leads to the following definition. 

\begin{definition}
 A function $g:[1,\infty) \rightarrow [0,\infty)$ is a {\em dimension-independent bounding function} if for every weakly dual systolic graph $G=(V,E)$ of dimension $d$, and every $s > 0$,
    \begin{equation*}
P_G(s) \geq d-g(s) .
    \end{equation*}
\end{definition}

The following lemma describes the underlying dynamics.
For $\epsilon > 0$, define a functional $\mathcal{F}_\epsilon$ by
$$\big[\mathcal{F}_\epsilon(g)\big](s) = g(2^\frac{4}{\epsilon})+\epsilon \log s.$$

\begin{lemma}\label{bootstrap_mechanism}
Let $g$ be a dimension-independent bounding function 
that is monotonically non-decreasing so that $g(2) \geq 1$.
Then, for every $\epsilon > 0$, the function
$\mathcal{F}_\epsilon(g)$ 
is also a dimension-independent bounding function.
\end{lemma}

\begin{proof}
Let $G$ be a $d$-dimensional weakly dual systolic graph,
and let $U$ be a set of vertices of size $s>0$.
We need to prove that
\begin{equation}\label{bootstrap_obj}
\phi(U) = \phi_{G}(U) \geq d - g(2^\frac{4}{\epsilon})-\epsilon \log s.
\end{equation}
The proof is by induction on $d$ and $s$.
There are two induction bases.
If $s=1$, then $\phi(U)=d$ and the inequality holds.
If $d=1$ and $1<s$, then the r.h.s.\ of~\eqref{bootstrap_obj} is non-positive.
It remains to perform the inductive step.
Let $V_1,\ldots,V_L$ be the connected components of the graph $G_{-d}$
obtained from $G$ by deleting all edges of color $d$.
Decompose $U$ to $U_1,\ldots,U_L$ defined by $U_\ell = U \cap V_\ell$.
Assume that only the first $U_1,\ldots,U_k$ are non-empty
(and the rest are empty).
If $k=1$, then induction on $d$ completes the proof,
because if $G'$ is the graph $G$ induces on $V_1$
then $\phi(U) \geq \phi_{G'}(U) + 1$ and $G'$ has dimension $d-1$.
So, we can assume that $k>1$. 
For each $i \in [k]$, we can 
consider $U_i$ and $U_{\neg i} = \bigcup_{j \neq i} U_{j}$. 
For convenience, let $p_i = \tfrac{|U_i|}{s}$ and let $q_i = 1-p_i$.
Both $U_i$ and $U_{\neg i}$ are smaller than $U$, 
so that induction implies that
\begin{align*}
\frac{\partial(U_i)}{p_i s} & \geq d-g(2^\frac{4}{\epsilon})-\epsilon \log(p_i s) ,  \\
 \frac{\partial(U_{\neg i})}{q_i s} & \geq   d-g(2^\frac{4}{\epsilon})-\epsilon \log(q_i s) .
\end{align*}
How does $\partial(U)$ relate to these quantities? 
The edges between $U_i$ and $U_{\neg i}$ are counted both in $\partial(U_i)$ and in $\partial(U_{\neg i})$.
All these edge have color $d$.
The number of edges between $U_i$ and $U_{\neg i}$ is denoted by $e(U_i,U_{\neg i})$.
Therefore,
\begin{equation*}
    \partial(U) = \partial(U_i) + \partial(U_{\neg i}) - 2e(U_i,U_{\neg i}) .
\end{equation*}
Combining the three former equations gives us:
\begin{align*}
\frac{\partial(U)}{s}  
& \geq d-g(2^\frac{4}{\epsilon})-\epsilon p_i \log(p_i s)-\epsilon q_i \log(q_is) - \frac{2e(U_i,U_{\neg i})}{s} \\
& = d-g(2^\frac{4}{\epsilon})-\epsilon \log(s) + \epsilon h(p_i) - \frac{2e(U_i,U_{\neg i})}{s} ,
\end{align*}
where $h$ is the binary entropy function.
The missing part for us can thus be summarized by the inequality:
\begin{equation}\label{density_ineq}
   \frac{2e(U_i,U_{\neg i})}{s} \leq \epsilon h(p_i) .
\end{equation}
Notice that we did not assume anything about the index $i$, and so it is enough to 
find any index for which~\eqref{density_ineq} is satisfied.
The analysis is partitioned between three cases.
\begin{center}
\textbf{Case I}: There exists $i \in [k]$ so that $p_i \leq 2^{-\frac{2}{\epsilon}}$.
\end{center}
In this case, we do not even need to use dual systolicity. Instead, we can use the bound $e(U_i,U_{\neg i}) \leq p_i s$
that holds due to proper coloring. Plugging this in~\eqref{density_ineq}, we get that it is enough to show $2p_i \leq  \epsilon h(p_i)$.
So, it is enough to show that
$2p_i \leq \epsilon p_i \log \frac{1}{p_i}$.
This is immediately satisfied by the assumption of the case.

\begin{center}
\textbf{Case II}: $2^{\frac{4}{\epsilon}} \leq s $ and for all $i\in [k]$, we have $2^{-\frac{2}{\epsilon}} < p_i$.
\end{center}
Because $\sum_{i \in [k]} p_i = 1$, we know that
$k \leq 2^{\frac{2}{\epsilon}}$.
Systolicity implies that $e(U_i,U_{\neg i}) \leq k$ for all $i$.
This means that it suffices to prove for some $i$ that
\begin{equation*}
    \frac{2 \cdot 2^\frac{2}{\epsilon}}{ s} \leq \epsilon h(p_i) .
\end{equation*}
Because $k>1$, we can choose $i$ so that $p_i \leq \frac{1}{2}$. 
This means we can treat the binary entropy as an increasing function, and lower bound it by $h(2^{-\frac{2}{\epsilon}})$. 
It suffices to show that
\begin{equation*}
    \frac{2 \cdot 2^\frac{2}{\epsilon}}{\epsilon s} \leq \frac{2}{\epsilon} \cdot 2^{-\frac{2}{\epsilon}} ,
\end{equation*}
which indeed holds.

\begin{center}
\textbf{Case III}: $s < 2^{\frac{4}{\epsilon}}$.
\end{center}

Because $g$ is monotonic non-decreasing, we have $g(s) \leq g(2^\frac{4}{\epsilon})$. So,
\begin{equation*}
\phi(U) \geq d-g(s) \geq d - g(2^\frac{4}{\epsilon}) - \epsilon \log s. \qedhere
\end{equation*}
\end{proof}

\subsection{The fixed point}

Lemma~\ref{bootstrap_mechanism} provides a bootstrapping mechanism.
We can start with an initial bounding function $g_0$.
The pseudo-cube bound allows use to choose $g_0(x) = \log x$.
From $g_0$, 
we can get a family of new bounds $g_{1,\epsilon} = \mathcal{F}_\epsilon(g_0)$ parametrized by~$\epsilon$. 
For every set-size $s$, we can choose an $\epsilon$ that gives the best possible bound for that $s$. This leads to a new bounding function $g_1$.
We can plug $g_1$ into the same mechanism and get a stronger bound $g_2$,
and so forth.
The fixed-point of this dynamics is the best possible bound
this proof gives (which turns out to be optimal up to constant factors).
The following lemma describes this mechanism.

\begin{lemma}\label{log_roots}
    Let $\ell > 0$ be an integer. 
    Suppose $g(s) = c \log^{\frac{1}{\ell}} s$ is a dimension-independent bounding function
    where $c\geq1$. Then, the following function is also a dimension-independent bounding function:
    \begin{equation*}
        g'(m) = c' \log^{\frac{1}{\ell+1}} s ,
    \end{equation*} 
    where
    \begin{equation}\label{new_coeff}
        c' = \big(4c^\ell \ell \big)^\frac{1}{\ell+1} + \Big(4 \Big(\frac{c}{\ell}\Big)^\ell \Big)^{\frac{1}{\ell+1}} \geq 1.
    \end{equation}
\end{lemma}
\begin{proof}
Plugging $g$ into $\mathcal{F}_\epsilon$ gives
    for every $\epsilon>0$, a new function
    \begin{align*}
        f_\epsilon(s) &= c \log^{\frac{1}{\ell}}(2^\frac{4}{\epsilon}) + \epsilon \log s\\
        &= c \Big(\frac{4}{\epsilon}\Big)^{\frac{1}{\ell}} + \epsilon \log s .
    \end{align*}
    For every $s$, we wish to find the best possible $\epsilon = \epsilon(s)$.
    An elementary calculation leads to the following choice
    \begin{align*}
        &\epsilon(s) = \Big(\frac{c 4^{\frac{1}{\ell}}}{\ell \log s}\Big)^{\frac{\ell}{\ell+1}} .
    \end{align*}
Plugging this $\epsilon$ gives
\begin{align*}
f'(s)
&  = c 4^{\frac{1}{\ell}} \Big(\frac{\ell \log s}{c 4^{\frac{1}{\ell}}}\Big)^{\frac{1}{\ell+1}}  + \Big(\frac{c 4^{\frac{1}{\ell}}}{\ell \log s}\Big)^{\frac{\ell}{\ell+1}} \log s\\
&  = \Big(c 4^{\frac{1}{\ell}} \Big(\frac{\ell}{c 4^{\frac{1}{\ell}}}\Big)^{\frac{1}{\ell+1}}  + \Big(\frac{c 4^{\frac{1}{\ell}}}{\ell}\Big)^{\frac{\ell}{\ell+1}} \Big) \log^{\frac{1}{\ell+1}} s .
\end{align*}
It remains to verify that $c' \geq 1$.
This follows by a direct substitution of $c=1$ 
that gives $(4\ell)^\frac{1}{\ell+1} + (\frac{4}{\ell^\ell})^\frac{1}{\ell+1} > 1$. 
\end{proof}

\begin{proof}[Proof of Theorem \ref{theta_loglog_weak}]
We build a sequence of dimension-independent bounding functions
$g_0,g_1,\ldots$ using Lemma~\ref{log_roots}.
    Theorem~\ref{thm:weakIso} allows to choose $g_0(s) = \log s$,
    which we can plug into Lemma~\ref{log_roots}.
    The lemma leads to a sequence of constants $c_0 =1$,
    and $c_{\ell+1}$ is defined from $c_\ell$ via~\eqref{new_coeff}.
A simple induction on $\ell$ shows that $c_\ell \leq 4\ell$. 
The induction base is true, and the induction step is
\begin{align*}
\big(4 (4\ell)^\ell \ell \big)^\frac{1}{\ell+1} + \Big(4 \Big(\frac{4\ell}{\ell}\Big)^\ell \Big)^{\frac{1}{\ell+1}} \leq 4 \ell + 4 .
\end{align*}
%
%
We thus have the following family of dimension-independent bounding functions:
\begin{equation}\label{l_dep_bound}
    g_\ell(s) = 4 \ell \log^{\frac{1}{\ell}} s .
\end{equation}
For each $s>1$, we can use $\ell$ that suits us best. 
Plugging $\ell = \lceil \log \log s \rceil$ gives
%
\begin{align*}
    g_\ell(s) 
    &= 4 \lceil \log \log s \rceil \log^{\frac{1}{\lceil \log \log s  \rceil}} s \\
    & \leq 4 ( 1 + \log \log s )  \log^{\frac{1}{ \log \log s}} s \\
    & = 8 ( 1 + \log \log s )  . \qedhere
\end{align*}
%
\end{proof}

\section{Threshold rank of clique products}
\begin{proof}[Proof of Theorem~\ref{thm:trhRk}]
Let $\epsilon = \frac{2k}{d}$ with
a positive integer $k \leq \tfrac{d}{2}$.
Let $M$ denote the normazlied adjacency matrix of $CP^{(d)}$.
Let $\lambda_1 \geq \lambda _2 \geq \ldots \geq \lambda_T$ be the eigenvalues of $M$ which are larger than $1-\epsilon$.
Let $u_1,\ldots,u_T$ be the corresponding normalized eigenvectors
(which are orthogonal). 
Denote by $U$ the span of $u_1,\ldots,u_T$.
The graph $CP^{(d)}$ is composed of multiple copies of $CP^{(d-k)}$. 
The number of copies is $\tfrac{n^{(d)}}{n^{(d-k)}}$.
For each copy $j$, let $\tilde v_j$ be the indictor vector
of the vertices in copy $j$,
and let $v_j = \frac{\tilde v_j}{\|\tilde v_j\|}$.
We claim two things:
\begin{enumerate}
\item For every $t$, 
$$\sum_{j} (\langle u_t, v_j \rangle)^2 \leq 1.$$
\item For every $j$, 
$$\sum_{t} (\langle u_t, v_j \rangle)^2 \geq \tfrac{1}{2}.$$
\end{enumerate}
The first item holds because the vectors $v_1,v_2,\ldots$
are orthonormal.     
    To prove the second item, fix $j$, and
    denote by $w_1$ the projection of $v_j$ to $U$
    and by $w_2$ the projection of $v_j$ to $U^\perp$.
   So, $\|w_1\|^2 = \sum_{t} (\langle u_t, v_j \rangle)^2$
   and $\|w_1\|^2 + \|w_2\|^2 = 1$.
The main point is that
\begin{align*}
\langle v_j , M v_j \rangle = 1-\tfrac{k}{d} .
       \end{align*}
The last calculation is
       \begin{align*}
        \langle v_j , M v_j \rangle 
        & = \langle w_1 , M w_1 \rangle + 2 \langle w_1 , M w_2 \rangle 
        + \langle w_2 , M w_2 \rangle \\
         & = \langle w_1 , M w_1 \rangle + \langle w_2 , M w_2 \rangle \\
         & \leq \|w_1\|^2 + \big(1- \tfrac{2k}{d} \big) (1-\| w_1 \|^2) \\
         & = 1- \tfrac{2k}{d} + \tfrac{2k}{d} \|w_1\|^2 ,
        \end{align*}
        so that
$\|w_1\|^2 \geq \frac{k}{d} \cdot \frac{d}{2k} = \frac{1}{2}$, as needed.

%

Now, taking the sum
    \begin{equation*}
\frac{n^{(d)}}{2n^{(d-k)}} \leq \sum_{j,t} (\langle u_t, v_j \rangle)^2 \leq T.
    \end{equation*}
Recalling the recurrence relation in Equation~\ref{simple_rec}, and applying it $k$ times, we have
$$\big(n^{(d-k)}\big)^{2^{k}} < n^{(d)}$$
so that
\begin{equation*}
T \geq \frac{\big( n^{(d)}\big)^{1-2^{-k}}}{2}. \qedhere
\end{equation*}
\end{proof}

\section*{Acknowledgement}

We wish to thank Irit Dinur, Nir Lazarovich, and Roy Meshulam
for helpful and insightful discussions. 


\bibliographystyle{abbrv}

\end{document}